\documentclass[12pt]{article}
\usepackage[a4paper,margin=1.2in]{geometry}
\usepackage{amssymb}
\usepackage{amsmath}
\usepackage{amsfonts}
\usepackage{amsthm}
\usepackage{color}
\usepackage{float}
\usepackage[all]{xy}
\usepackage{scalerel}
\usepackage{mathabx}
\usepackage{array}
\usepackage{float}
\usepackage{delimset}

\newcommand{\qbinom}{\genfrac{[}{]}{0pt}{}}

\newcommand{\e}{\mathrm{e}}
\newcommand{\E}{\mathrm{E}}

\newcommand{\s}{\mathrm{S}^{*}}

\newcommand{\RR}{\mathcal{R}}
\newcommand{\rr}{\mathcal{R}}
\newcommand{\rrr}{\mathrm{r}}

\newtheorem{theorem}{Theorem}

\newtheorem{definition}{Definition}

\begin{document}

\title{\textbf{Rogers-Ramanujan Type $q$-Exponential Operator and Stieljes-Widgert Polynomials}}
\author{Ronald Orozco L\'opez}

\newcommand{\Addresses}{{
  \bigskip
  \footnotesize

  \textit{E-mail address}, R.~Orozco: \texttt{rj.orozco@uniandes.edu.co}
  
}}

\maketitle


\begin{abstract}
In this paper, we use the Rogers-Ramanujan type $q$-exponential operator $\rr(qD_{q})$ to derive generating functions, and Mehler and Rogers formulas, for the non-normalized homogeneous Stieljes-Wigert polynomials $\s_{n}(x,y;q)$.
\end{abstract}

\section{Introduction}

We begin with some notation and terminology for basic hypergeometric series \cite{gasper}. Let $\vert q\vert<1$ and the $q$-shifted factorial be defined by
\begin{align*}
    (a;q)_{0}&=1,\\
    (a;q)_{n}&=\prod_{k=0}^{n-1}(1-q^{k}a),\\
    (a;q)_{\infty}&=\lim_{n\rightarrow\infty}(a;q)_{n}=\prod_{k=0}^{\infty}(1-aq^{k}).
\end{align*}
The multiple $q$-shifted factorials are defined by
\begin{align*}
    (a_{1},a_{2},\ldots,a_{m};q)_{n}&=(a_{1};q)_{n}(a_{2};q)_{n}\cdots(a_{m};q)_{n},\\
    (a_{1},a_{2},\ldots,a_{m};q)_{\infty}&=(a_{1};q)_{\infty}(a_{2};q)_{\infty}\cdots(a_{m};q)_{\infty}.
\end{align*}
Some useful identities for $q$-shifted factorial:
\begin{align}
    (a;q)_{n}&=\frac{(a;q)_{\infty}}{(aq^n;q)_{\infty}},\label{eqn_iden1}\\
    (a;q)_{n+k}&=(a;q)_{n}(aq^{n};q)_{k},\label{eqn_iden2}\\
    (aq^n;q)_{k}&=\frac{(a;q)_{k}(aq^k;q)_{n}}{(a;q)_{n}},\label{eqn_iden3}
\end{align}
In our work, we will use the identities for binomial coefficients:
\begin{align*}
    \binom{n+k}{2}&=\binom{n}{2}+\binom{k}{2}+nk,\\
    \binom{n-k}{2}&=\binom{n}{2}+\binom{k}{2}+k(1-n).
\end{align*}
The $q$-binomial coefficient is defined by
\begin{equation*}
\qbinom{n}{k}_{q}=\frac{(q;q)_{n}}{(q;q)_{k}(q;q)_{n-k}}.
\end{equation*}
The ${}_r\phi_{s}$ basic hypergeometric series is define by
\begin{equation*}
    {}_r\phi_{s}\left(
    \begin{array}{c}
         a_{1},a_{2},\ldots,a_{r} \\
         b_{1},\ldots,b_{s}
    \end{array}
    ;q,z
    \right)=\sum_{n=0}^{\infty}\frac{(a_{1},a_{2},\ldots,a_{r};q)_{n}}{(q;q)_{n}(b_{1},b_{2},\ldots,b_{s};q)_{n}}\Big[(-1)^{n}q^{\binom{n}{2}}\Big]^{1+s-r}z^n.
\end{equation*}
In this paper, we will frequently use the $q$-binomial theorem:
\begin{equation}\label{eqn_qbin_the}
    {}_1\phi_{0}(a;q,z)=\frac{(az;q)_{\infty}}{(z;q)_{\infty}}=\sum_{n=0}^{\infty}\frac{(a;q)_{n}}{(q;q)_{n}}z^{n}.
\end{equation}
The $q$-exponential $\e_{q}(z)$ is defined by
\begin{equation*}
    \e_{q}(z)=\sum_{n=0}^{\infty}\frac{z^n}{(q;q)_{n}}={}_1\phi_{0}\left(\begin{array}{c}
         0\\
         - 
    \end{array};q,-z\right)=\frac{1}{(z;q)_{\infty}}.
\end{equation*}
Another $q$-analogue of the classical exponential function is
\begin{equation*}
    \E_{q}(z)=\sum_{n=0}^{\infty}q^{\binom{n}{2}}\frac{z^n}{(q;q)_{n})}={}_1\phi_{1}\left(\begin{array}{c}
         0\\
         0 
    \end{array};q,-z\right)=(-z;q)_{\infty}.
\end{equation*}
The Stieljes-Wigert polynomials are defined by
\begin{equation}\label{eqn_SWP}
    \mathrm{S}_{n}(x;q)=\frac{1}{(q;q)_{n}}{}_{1}\phi_{1}\left(\begin{array}{c}
         q^{-n}\\
         0
    \end{array};q,-q^{n+1}x\right).
\end{equation}
Some known generating functions are:
\begin{align}
    \frac{1}{(t;q)_{\infty}}{}_{0}\phi_{1}\left(\begin{array}{c}
         -\\
         0
    \end{array};q,-qtx\right)&=\sum_{n=0}^{\infty}\mathrm{S}_{n}(x;q)t^n.\label{eqn_GF1}\\
    (t;q)_{\infty}\cdot{}_{0}\phi_{2}\left(\begin{array}{c}
         -\\
         0,t
    \end{array};q,-qtx\right)&=\sum_{n=0}^{\infty}(-1)^nq^{\binom{n}{2}}\mathrm{S}_{n}(x;q)t^n.\label{eqn_GF2}\\
    \frac{(at;q)_{\infty}}{(t;q)_{\infty}}{}_{1}\phi_{2}\left(\begin{array}{c}
         a\\
         0,at
    \end{array};q,-qtx\right)&=\sum_{n=0}^{\infty}(a;q)_{n}\mathrm{S}_{n}(x;q)t^n.\label{eqn_GF3}
\end{align}
In this paper, we define the non-normalized homogeneous Stieljes-Wigert polynomial of order $n$
\begin{equation}
    \s_{n}(x,y;q)=\sum_{n=0}^{\infty}\qbinom{n}{k}_{q}q^{k^2}x^{n-k}y^k
\end{equation}
and represent it through the Rogers-Ramanujan type $q$-exponential operator \cite{orozco}
\begin{equation}
    \rr(yD_{q})=\sum_{n=0}^{\infty}q^{n^2}\frac{y^n}{(q;q)_{n}}D_{q}^n,
\end{equation}
to derive its generating functions, Mehler's and Rogers's type formulas.

\section{$q$-differential operator}

The $q$-differential operator $D_{q}$ is defined by:
\begin{equation*}
    D_{q}f(x)=\frac{f(x)-f(qx)}{x}
\end{equation*}
and the Leibniz rule for $D_{q}$
\begin{equation}\label{eqn_leibniz}
    D_{q}^{n}\{f(x)g(x)\}=\sum_{k=0}^{n}q^{k(k-n)}\qbinom{n}{k}_{q}D_{q}^{k}\{f(x)\}D_{q}^{n-k}\{g(q^{k}x)\}.
\end{equation}
Then
\begin{align}
    D_{q}^nx^{k}&=\frac{(q;q)_{k}}{(q;q)_{k-n}}x^{k-n},\label{eqn_iden4}\\
    D_{q}^n\left\{\frac{1}{(ax;q)_{\infty}}\right\}&=\frac{a^n}{(ax;q)_{\infty}},\label{eqn_iden5}\\
    D_{q}^{n}\{(ax;q)_{\infty}\}&=(-a)^{n}q^{\binom{n}{2}}(aq^nx;q)_{\infty}.\label{eqn_iden6}
\end{align}
From the Leibniz formula Eq.(\ref{eqn_leibniz}) and the Eqs.(\ref{eqn_iden1}), (\ref{eqn_iden5}), and (\ref{eqn_iden6}), then
    \begin{align}
        D_{q}^{n}\{(ax,bx;q)_{\infty}\}&=q^{\binom{n}{2}}(ax,bq^nx;q)_{\infty}\sum_{k=0}^{n}\qbinom{n}{k}_{q}q^{k(k-n)}\frac{a^kb^{n-k}}{(ax;q)_{k}},\label{eqn_iden7}\\
        D_{q}^{n}\left\{\frac{(ax;q)_{\infty}}{(bx;q)_{\infty}}\right\}&=\frac{(ax;q)_{\infty}}{(bx;q)_{\infty}}\sum_{k=0}^{n}\qbinom{n}{k}_{q}q^{\binom{k}{2}}(-a)^kb^{n-k}\frac{(bx;q)_{k}}{(ax;q)_{k}},\label{eqn_iden8}\\
        D_{q}^{n}\left\{\frac{1}{(ax,bx;q)_{\infty}}\right\}
        &=\frac{1}{(ax,bx;q)_{\infty}}\sum_{k=0}^{n}\qbinom{n}{k}_{q}a^kb^{n-k}(bx;q)_{k}.\label{eqn_iden9}
    \end{align}

\section{The Ramanujan $q$-exponential function}

\begin{definition}
The Ramanujan type $q$-exponential function $\RR_{q}(z)$ is defined as
\begin{equation}
    \RR_{q}(z)=\sum_{n=0}^{\infty}q^{n^2}\frac{z^n}{(q;q)_{n}}.
\end{equation}
\end{definition}
The function $\RR_{q}(z)$ satisfies the following equation in differences
\begin{equation}
    \RR_{q}(z)-\RR_{q}(qz)=(1-q)z\RR_{q}(q^2z).
\end{equation}
The $n$-th $q$-derivative of $\RR_{q}(z)$ is
\begin{equation}
    D_{q}^{n}\{\RR_{q}(az)\}=a^{n}q^{n^2}\RR_{q}(aq^{2n}z).
\end{equation}
Some special values of the function $\RR_{q}(z)$ are:
\begin{equation}
    \RR_{q}(1)=\frac{1}{(q,q^4;q^5)_{\infty}}
\end{equation}
and
\begin{equation}
    \RR_{q}(q)=\frac{1}{(q^2,q^3;q^5)_{\infty}}.
\end{equation}
From Garrett \cite{garrett}
\begin{align}\label{eqn_eqn_garrett}
    \RR_{q}(q^{k})=\sum_{n=0}^{\infty}\frac{q^{n^2+kn}}{(q;q)_{n}}&=\RR_{q}(1)q^{-\binom{k}{2}}a_{k}(q)-\RR_{q}(q)q^{-\binom{k}{2}}b_{k}(q),
\end{align}
where
\begin{equation*}
    a_{k}(q)=\sum_{i}(-1)^iq^{i(5i-3)/2}\qbinom{k-1}{\lfloor\frac{k+1-5i}{2}\rfloor}_{q}
\end{equation*}
and
\begin{equation*}
    b_{k}(q)=\sum_{i}(-1)^iq^{i(5i+1)/2}\qbinom{k-1}{\lfloor\frac{k-1-5i}{2}\rfloor}_{q},
\end{equation*}
where the initial values are defined as $a_{0}=1$ and $b_{0}=0$.

\section{Generating functions of Stieljes-Widgert polynomials}

\begin{definition}\label{def_translation}
Define the $q$-exponential operator type Rogers-Ramanujan $\RR(yD_{q})$ by letting 
\begin{equation*}
    \rr(yD_{q})=\sum_{n=0}^{\infty}q^{n^2}\frac{(yD_{q})^{n}}{(q;q)_{n}}. 
\end{equation*}
\end{definition}

\begin{definition}
We define the bivariate Stieljes-Wigert polynomials as
    \begin{equation*}
        \s_{n}(x,y;q)=\sum_{k=0}^{n}\qbinom{n}{k}_{q}q^{k^2}x^{n-k}y^k.
    \end{equation*}
\end{definition}

\begin{theorem}
    \begin{equation*}
        \rr(yD_{q})\{x^n\}=\s_{n}(x,y;q).
    \end{equation*}
\end{theorem}
\begin{proof}
    \begin{align*}
        \rr(yD_{q})\{x^n\}&=\sum_{k=0}^{\infty}q^{k^2}\frac{y^k}{(q;q)_{k}}D_{q}^k\{x^n\}=\sum_{k=0}^{n}\qbinom{n}{k}_{q}q^{k^2}y^kx^{n-k}.
    \end{align*}
\end{proof}

\begin{theorem}
    \begin{equation*}
        \rr(yD_{q})\left\{\frac{1}{(ax;q)_{\infty}}\right\}=\frac{\RR_{q}(ay)}{(ax;q)_{\infty}}.
    \end{equation*}
\end{theorem}
\begin{proof}
    \begin{align*}
        \rr(yD_{q})\left\{\frac{1}{(ax;q)_{\infty}}\right\}&=\sum_{k=0}^{\infty}q^{k^2}\frac{y^k}{(q;q)_{k}}D_{q}^k\left\{\frac{1}{(ax;q)_{\infty}}\right\}\\
        &=\sum_{k=0}^{\infty}q^{k^2}\frac{y^k}{(q;q)_{k}}\frac{a^k}{(ax;q)_{\infty}}\\
        &=\frac{\RR_{q}(ay)}{(ax;q)_{\infty}}.
    \end{align*}
\end{proof}

\begin{theorem}
    \begin{equation*}
        \sum_{n=0}^{\infty}\s_{n}(x,y;q)\frac{z^n}{(q;q)_{n}}=\frac{\RR_{q}(zy)}{(zx;q)_{\infty}}.
    \end{equation*}
\end{theorem}
\begin{proof}
    \begin{align*}
        \sum_{n=0}^{\infty}\s_{n}(x,y;q)\frac{z^n}{(q;q)_{n}}&=\sum_{n=0}^{\infty}\frac{\rr(yD_{q})\{(zx)^n\}}{(q;q)_{n}}\\
        &=\rr(yD_{q})\left\{\sum_{n=0}^{\infty}\frac{(zx)^n}{(q;q)_{n}}\right\}\\
        &=\rr(yD_{q})\left\{\frac{1}{(zx;q)_{\infty}}\right\}\\
        &=\frac{\RR_{q}(zy)}{(zx;q)_{\infty}}.
    \end{align*}
\end{proof}

\begin{theorem}
    \begin{equation*}
        \rr(yD_{q})\left\{(ax;q)_{\infty}\right\}=(ax;q)_{\infty}\cdot{}_{0}\phi_{2}\left(
    \begin{array}{c}
         -\\
         ax,0
    \end{array}
    ;q,qay
    \right).
    \end{equation*}
\end{theorem}
\begin{proof}
    \begin{align*}
        \rr(yD_{q})\left\{(ax;q)_{\infty}\right\}&=\sum_{n=0}^{\infty}q^{n^2}\frac{y^n}{(q;q)_{n}}D_{q}^{n}\left\{(ax;q)_{\infty}\right\}\\
        &=(ax;q)_{\infty}\sum_{n=0}^{\infty}q^{n^2}\frac{y^n}{(q;q)_{n}}\frac{(-a)^nq^{\binom{n}{2}}}{(ax;q)_{n}}\\
        &=(ax;q)_{\infty}\cdot{}_{0}\phi_{2}\left(
    \begin{array}{c}
         -\\
         ax,0
    \end{array}
    ;q,qay
    \right).
    \end{align*}
\end{proof}

\begin{theorem}
    \begin{equation*}
        \sum_{n=0}^{\infty}(-1)^nq^{\binom{n}{2}}\s_{n}(x,y;q)\frac{z^n}{(q;q)_{n}}=(zx;q)_{\infty}\cdot{}_{0}\phi_{2}\left(
    \begin{array}{c}
         -\\
         zx,0
    \end{array}
    ;q,qzy
    \right).
    \end{equation*}
\end{theorem}
\begin{proof}
    \begin{align*}
        \sum_{n=0}^{\infty}(-1)^nq^{\binom{n}{2}}\s_{n}(x,y;q)\frac{z^n}{(q;q)_{n}}&=\sum_{n=0}^{\infty}(-1)^nq^{\binom{n}{2}}\frac{\rr(yD_{q})\{(zx)^n\}}{(q;q)_{n}}\\
        &=\rr(yD_{q})\left\{\sum_{n=0}^{\infty}(-1)^nq^{\binom{n}{2}}\frac{(zx)^n}{(q;q)_{n}}\right\}\\
        &=\rr(yD_{q})\left\{(zx;q)_{\infty}\right\}\\
        &=(zx;q)_{\infty}\cdot{}_{0}\phi_{2}\left(
    \begin{array}{c}
         -\\
         zx,0
    \end{array}
    ;q,qzy
    \right).
    \end{align*}
\end{proof}

\begin{theorem}
    \begin{equation*}
        \rr(yD_{q})\left\{\frac{(az;q)_{\infty}}{(z;q)_{\infty}}\right\}=\frac{(az;q)_{\infty}}{(z;q)_{\infty}}{}_{1}\phi_{2}\left(
    \begin{array}{c}
         a\\
         az,0
    \end{array}
    ;q,qy
    \right).
    \end{equation*}
\end{theorem}
\begin{proof}
From $q$-binomial theorem,
    \begin{align*}
        \rr(yD_{q})\left\{\frac{(az;q)_{\infty}}{(z;q)_{\infty}}\right\}
        &=\sum_{n=0}^{\infty}q^{n^2}\frac{y^n}{(q;q)_{n}}D_{q}^n\left\{\frac{(az;q)_{\infty}}{(z;q)_{\infty}}\right\}\\
        &=\sum_{n=0}^{\infty}q^{n^2}\frac{y^n}{(q;q)_{n}}D_{q}^n\left\{{}_{1}\phi_{0}\left(
    \begin{array}{c}
         a\\
         -
    \end{array}
    ;q,z
    \right)\right\}\\
    &=\sum_{n=0}^{\infty}q^{n^2}\frac{y^n}{(q;q)_{n}}(a;q)_{n}\cdot{}_{1}\phi_{0}\left(
    \begin{array}{c}
         aq^n\\
         -
    \end{array}
    ;q,z
    \right)\\
    &=\sum_{n=0}^{\infty}q^{n^2}\frac{(a;q)_{n}}{(q;q)_{n}}y^{n}\frac{(aq^nz;q)_{\infty}}{(z;q)_{\infty}}\\
    &=\frac{(az;q)_{\infty}}{(z;q)_{\infty}}\sum_{n=0}^{\infty}q^{n^2}\frac{(a;q)_{n}}{(q;q)_{n}(az;q)_{n}}y^{n}\\
    &=\frac{(az;q)_{\infty}}{(z;q)_{\infty}}{}_{1}\phi_{2}\left(
    \begin{array}{c}
         a\\
         az,0
    \end{array}
    ;q,qy
    \right).
    \end{align*}
\end{proof}

\begin{theorem}
If $by=1$, then
    \begin{align*}
        &\rr(yD_{q})\left\{\frac{(ax;q)_{\infty}}{(xb;q)_{\infty}}\right\}\\
        &\hspace{1cm}=\frac{(ax;q)_{\infty}}{(bx;q)_{\infty}(q,q^4;q^5)_{\infty}}\sum_{k=0}^{\infty}q^{-\binom{k}{2}}\frac{(bx;q)_{k}a_{2k}(q)}{(ax;q)_{k}(q;q)_{k}}(-ay)^k\\
    &\hspace{2cm}-\frac{(ax;q)_{\infty}}{(bx;q)_{\infty}(q^2,q^3;q^5)_{\infty}}\sum_{k=0}^{\infty}q^{-\binom{k}{2}}\frac{(bx;q)_{k}b_{2k}(q)}{(ax;q)_{k}(q;q)_{k}}(-ay)^k.
    \end{align*}
\end{theorem}
\begin{proof}
    \begin{align*}
        &\rr(yD_{q})\left\{\frac{(ax;q)_{\infty}}{(bx;q)_{\infty}}\right\}
        =\sum_{n=0}^{\infty}q^{n^2}\frac{y^n}{(q;q)_{n}}D_{q}^n\left\{\frac{(ax;q)_{\infty}}{(bx;q)_{\infty}}\right\}\\
        &=\sum_{n=0}^{\infty}q^{n^2}\frac{y^n}{(q;q)_{n}}\sum_{k=0}^{n}\qbinom{n}{k}_{q}q^{k(k-n)}D_{q}^{k}\{(ax;q)_{\infty}\}D_{q}^{n-k}\left\{\frac{1}{(bq^{k}x;q)_{\infty}}\right\}\\
        &=\frac{(ax;q)_{\infty}}{(bx;q)_{\infty}}\sum_{n=0}^{\infty}q^{n^2}\frac{y^n}{(q;q)_{n}}\sum_{k=0}^{n}\qbinom{n}{k}_{q}q^{\binom{k}{2}}\frac{(bx;q)_{k}}{(ax;q)_{k}}(-a)^kb^{n-k}\\
        &=\frac{(ax;q)_{\infty}}{(bx;q)_{\infty}}\sum_{k=0}^{\infty}q^{\binom{k}{2}}\frac{q^{k^2}(bx;q)_{k}}{(ax;q)_{k}(q;q)_{k}}(-ay)^k\sum_{n=0}^{\infty}q^{n^2}\frac{(q^{2k}by)^{n}}{(q;q)_{n}}\\
        &=\frac{(ax;q)_{\infty}}{(bx;q)_{\infty}}\sum_{k=0}^{\infty}q^{(3k^2-k)/2}\frac{(bx;q)_{k}}{(ax;q)_{k}(q;q)_{k}}(-ay)^k\RR_{q}(q^{2k}by).
    \end{align*}
If $by=1$, then from Eq(\ref{eqn_eqn_garrett})    
\begin{align*}
    &\rr(yD_{q})\left\{\frac{(ax;q)_{\infty}}{(bx;q)_{\infty}}\right\}\\
    &\hspace{1cm}=\frac{(ax;q)_{\infty}}{(bx;q)_{\infty}}\sum_{k=0}^{\infty}q^{(3k^2-k)/2}\frac{(bx;q)_{k}}{(ax;q)_{k}(q;q)_{k}}(-ay)^k\\
    &\hspace{4cm}\times(\RR_{q}(1)q^{-\binom{2k}{2}}a_{2k}(q)-\RR_{q}(q)q^{-\binom{2k}{2}}b_{2k}(q))\\
    &\hspace{1cm}=\frac{(ax;q)_{\infty}}{(bx;q)_{\infty}(q,q^4;q^5)_{\infty}}\sum_{k=0}^{\infty}q^{-\binom{k}{2}}\frac{(bx;q)_{k}a_{2k}(q)}{(ax;q)_{k}(q;q)_{k}}(-ay)^k\\
    &\hspace{2cm}-\frac{(ax;q)_{\infty}}{(bx;q)_{\infty}(q^2,q^3;q^5)_{\infty}}\sum_{k=0}^{\infty}q^{-\binom{k}{2}}\frac{(bx;q)_{k}b_{2k}(q)}{(ax;q)_{k}(q;q)_{k}}(-ay)^k.
\end{align*}
\end{proof}

\begin{theorem} The Srivastava-Agarwal type representation \cite{Sri} of $\s_{n}(x,y;q)$ is
    \begin{equation*}
        \sum_{n=0}^{\infty}\s_{n}(x,y;q)\frac{(a;q)_{n}}{(q;q)_{n}}z^n=\frac{(azx;q)_{\infty}}{(zx;q)_{\infty}}{}_{1}\phi_{2}\left(
    \begin{array}{c}
         a\\
         azx,0
    \end{array}
    ;q,qy
    \right).
    \end{equation*}    
If $yz=1$, then
\begin{align*}
    &\sum_{n=0}^{\infty}\s_{n}(x,y;q)\frac{(a;q)_{n}}{(q;q)_{n}}z^n\\
    &\hspace{1cm}=\frac{(azx;q)_{\infty}}{(zx;q)_{\infty}(q,q^4;q^5)_{\infty}}\sum_{k=0}^{\infty}q^{-\binom{k}{2}}\frac{(zx;q)_{k}a_{2k}(q)}{(azx;q)_{k}(q;q)_{k}}(-azy)^k\\
    &\hspace{2cm}-\frac{(azx;q)_{\infty}}{(zx;q)_{\infty}(q^2,q^3;q^5)_{\infty}}\sum_{k=0}^{\infty}q^{-\binom{k}{2}}\frac{(zx;q)_{k}b_{2k}(q)}{(azx;q)_{k}(q;q)_{k}}(-azy)^k.
\end{align*}
\end{theorem}
\begin{proof}
    \begin{align*}
        \sum_{n=0}^{\infty}\s_{n}(x,y;q)\frac{(a;q)_{n}}{(q;q)_{n}}z^n&=\sum_{n=0}^{\infty}\frac{(a;q)_{n}}{(q;q)_{n}}\rr(yD_{q})\left\{(zx)^n\right\}\\
        &=\rr(yD_{q})\left\{\sum_{n=0}^{\infty}\frac{(a;q)_{n}}{(q;q)_{n}}(zx)^n\right\}\\
        &=\rr(yD_{q})\left\{\frac{(azx;q)_{\infty}}{(zx;q)_{\infty}}\right\}\\
    &=\frac{(azx;q)_{\infty}}{(zx;q)_{\infty}}{}_{1}\phi_{2}\left(
    \begin{array}{c}
         a\\
         azx,0
    \end{array}
    ;q,qy
    \right).
    \end{align*}
If $yz=1$, then from Eq.(\ref{eqn_eqn_garrett})
\begin{align*}
    &\sum_{n=0}^{\infty}\s_{n}(x,y;q)\frac{(a;q)_{n}}{(q;q)_{n}}z^n\\
    &\hspace{1cm}=\frac{(azx;q)_{\infty}}{(zx;q)_{\infty}(q,q^4;q^5)_{\infty}}\sum_{k=0}^{\infty}q^{-\binom{k}{2}}\frac{(zx;q)_{k}a_{2k}(q)}{(azx;q)_{k}(q;q)_{k}}(-azy)^k\\
    &\hspace{2cm}-\frac{(azx;q)_{\infty}}{(zx;q)_{\infty}(q^2,q^3;q^5)_{\infty}}\sum_{k=0}^{\infty}q^{-\binom{k}{2}}\frac{(zx;q)_{k}b_{2k}(q)}{(azx;q)_{k}(q;q)_{k}}(-azy)^k.
\end{align*}
\end{proof}

\begin{theorem}
\begin{align*}
    \rr(yD_{q})\left\{\frac{1}{(ax,bx;q)_{\infty}}\right\}=\frac{1}{(ax,bx;q)_{\infty}}\sum_{i=0}^{\infty}q^{i^2}\frac{(bx;q)_{i}}{(q;q)_{i}}(ay)^i\RR_{q}(bq^{2i}y).
\end{align*}
If $by=1$, then
    \begin{align*}
        &\rr(yD_{q})\left\{\frac{1}{(ax,bx;q)_{\infty}}\right\}\\
        &\hspace{1cm}=\frac{1}{(ax,bx;q)_{\infty}(q,q^4;q^5)_{\infty}}\sum_{i=0}^{\infty}q^{-i(3i-2)/2}\frac{(bx;q)_{i}}{(q;q)_{i}}a_{2i}(q)(ay)^i\\
    &\hspace{2cm}-\frac{1}{(ax,bx;q)_{\infty}(q^2,q^3;q^5)_{\infty}}\sum_{i=0}^{\infty}q^{-i(3-2)/2}\frac{(bx;q)_{i}}{(q;q)_{i}}b_{2i}(q)(ay)^i.
    \end{align*}
\end{theorem}
\begin{proof}
    \begin{align*}
        \rr(yD_{q})\left\{\frac{1}{(ax,bx;q)_{\infty}}\right\}&=\sum_{k=0}^{\infty}q^{k^2}\frac{y^k}{(q;q)_{k}}D_{q}^k\left\{\frac{1}{(ax,bx;q)_{\infty}}\right\}\\
        &=\frac{1}{(ax,bx;q)_{\infty}}\sum_{k=0}^{\infty}q^{k^2}\frac{y^k}{(q;q)_{k}}\sum_{i=0}^{k}\qbinom{k}{i}_{q}a^ib^{k-i}(bx;q)_{i}\\
        &=\frac{1}{(ax,bx;q)_{\infty}}\sum_{i=0}^{\infty}\frac{(bx;q)_{i}}{(q;q)_{i}}a^i\sum_{k=i}^{\infty}q^{k^2}\frac{y^k}{(q;q)_{k-i}}b^{k-i}\\
        &=\frac{1}{(ax,bx;q)_{\infty}}\sum_{i=0}^{\infty}q^{i^2}\frac{(bx;q)_{i}}{(q;q)_{i}}(ay)^i\sum_{k=0}^{\infty}q^{k^2}\frac{(bq^{2i}y)^{k}}{(q;q)_{k}}\\
        &=\frac{1}{(ax,bx;q)_{\infty}}\sum_{i=0}^{\infty}q^{i^2}\frac{(bx;q)_{i}}{(q;q)_{i}}(ay)^i\RR_{q}(bq^{2i}y).
    \end{align*}
If $by=1$, then from Eq.(\ref{eqn_eqn_garrett})
\begin{align*}
    &\rr(yD_{q})\left\{\frac{1}{(ax,bx;q)_{\infty}}\right\}\\
    &=\frac{1}{(ax,bx;q)_{\infty}}\sum_{i=0}^{\infty}q^{i^2}\frac{(bx;q)_{i}}{(q;q)_{i}}(ay)^i(\RR_{q}(1)q^{-\binom{2i}{2}}a_{2i}(q)-\RR_{q}(q)q^{-\binom{2i}{2}}b_{2i}(q))\\
    &=\frac{1}{(ax,bx;q)_{\infty}(q,q^4;q^5)_{\infty}}\sum_{i=0}^{\infty}q^{-i(3i-2)/2}\frac{(bx;q)_{i}}{(q;q)_{i}}a_{2i}(q)(ay)^i\\
    &\hspace{2cm}-\frac{1}{(ax,bx;q)_{\infty}(q^2,q^3;q^5)_{\infty}}\sum_{i=0}^{\infty}q^{-i(3-2)/2}\frac{(bx;q)_{i}}{(q;q)_{i}}b_{2i}(q)(ay)^i.
\end{align*}
\end{proof}

\begin{theorem}
\begin{align*}
    \sum_{n=0}^{\infty}\s_{n}(x,y;q)\rrr_{n}(a,b)\frac{z^n}{(q;q)_{n}}=\frac{1}{(azx,bzx;q)_{\infty}}\sum_{i=0}^{\infty}q^{i^2}\frac{(bzx;q)_{i}}{(q;q)_{i}}(ay)^i\RR_{q}(bq^{2i}y)
\end{align*}
and if $bzy=1$, then
\begin{align*}
    &\sum_{n=0}^{\infty}\s_{n}(x,y;q)\rrr_{n}(a,b)\frac{z^n}{(q;q)_{n}}\\
    &\hspace{1cm}=\frac{1}{(azx,bzx;q)_{\infty}(q,q^4;q^5)_{\infty}}\sum_{i=0}^{\infty}q^{-i(3i-2)/2}\frac{(bzx;q)_{i}}{(q;q)_{i}}a_{2i}(q)(azy)^i\\
    &\hspace{2cm}-\frac{1}{(azx,bzx;q)_{\infty}(q^2,q^3;q^5)_{\infty}}\sum_{i=0}^{\infty}q^{-i(3-2)/2}\frac{(bzx;q)_{i}}{(q;q)_{i}}b_{2i}(q)(azy)^i,
\end{align*}    
where
\begin{equation}
    \rrr_{n}(a,b)=\sum_{k=0}^{\infty}\qbinom{n}{k}_{q}a^{n-k}b^k
\end{equation}
is the generalized Rogers-Szeg\"o polynomials \cite{saad}.
\end{theorem}
\begin{proof}
    \begin{align*}
        \sum_{n=0}^{\infty}\s_{n}(x,y;q)\rrr_{n}(a,b)\frac{z^n}{(q;q)_{n}}&=\sum_{n=0}^{\infty}\rr(yD_{q})\{x^n\}\rrr_{n}(a,b)\frac{z^n}{(q;q)_{n}}\\
        &=\rr(yD_{q})\left\{\sum_{n=0}^{\infty}\rrr_{n}(a,b)\frac{(zx)^n}{(q;q)_{n}}\right\}\\
        &=\rr(yD_{q})\left\{\frac{1}{(azx,bzx;q)_{\infty}}\right\}\\
        &=\frac{1}{(azx,bzx;q)_{\infty}}\sum_{i=0}^{\infty}q^{i^2}\frac{(bzx;q)_{i}}{(q;q)_{i}}(ay)^i\RR_{q}(bq^{2i}y).
    \end{align*}
If $by=1$, then
\begin{align*}
    &\sum_{n=0}^{\infty}\s_{n}(x,y;q)\rrr_{n}(a,b)\frac{z^n}{(q;q)_{n}}\\
    &\hspace{1cm}=\frac{1}{(azx,bzx;q)_{\infty}(q,q^4;q^5)_{\infty}}\sum_{i=0}^{\infty}q^{-i(3i-2)/2}\frac{(bzx;q)_{i}}{(q;q)_{i}}a_{2i}(q)(azy)^i\\
        &\hspace{2cm}-\frac{1}{(azx,bzx;q)_{\infty}(q^2,q^3;q^5)_{\infty}}\sum_{i=0}^{\infty}q^{-i(3-2)/2}\frac{(bzx;q)_{i}}{(q;q)_{i}}b_{2i}(q)(azy)^i.
\end{align*}
\end{proof}

\begin{theorem}
    \begin{equation*}
        \sum_{n=0}^{\infty}\s_{n}(x,y;q)\frac{(a;q)_{n}}{(b;q)_{n}(q;q)_{n}}z^n=\frac{(a;q)_{\infty}}{(zx,b;q)_{\infty}}\sum_{k=0}^{\infty}\frac{(b/a;q)_{k}(zx;q)_{k}}{(q;q)_{k}}a^k\RR_{q}(q^kzy).
    \end{equation*}
\end{theorem}
\begin{proof}
From the $q$-theorem binomial,
\begin{equation*}
    \frac{(bq^n;q)_{\infty}}{(aq^n;q)_{\infty}}=\sum_{k=0}^{\infty}\frac{(b/a;q)_{k}}{(q;q)_{k}}(aq^n)^k.
\end{equation*}
Hence
    \begin{align*}
        \sum_{n=0}^{\infty}\s_{n}(x,y;q)\frac{(a;q)_{n}}{(b;q)_{n}(q;q)_{n}}z^n&=\frac{(a;q)_{\infty}}{(b;q)_{\infty}}\sum_{n=0}^{\infty}\s_{n}(x,y;q)\frac{(bq^n;q)_{n}}{(aq^n;q)_{n}(q;q)_{n}}z^n\\
        &=\frac{(a;q)_{\infty}}{(b;q)_{\infty}}\sum_{n=0}^{\infty}\s_{n}(x,y;q)\frac{z^n}{(q;q)_{n}}\sum_{k=0}^{\infty}\frac{(b/a;q)_{k}}{(q;q)_{k}}(aq^n)^k\\
        &=\frac{(a;q)_{\infty}}{(b;q)_{\infty}}\sum_{k=0}^{\infty}\frac{(b/a;q)_{k}}{(q;q)_{k}}a^k\sum_{n=0}^{\infty}\s_{n}(x,y;q)\frac{(q^kz)^n}{(q;q)_{n}}\\
        &=\frac{(a;q)_{\infty}}{(b;q)_{\infty}}\sum_{k=0}^{\infty}\frac{(b/a;q)_{k}}{(q;q)_{k}}a^k\frac{\RR_{q}(q^kzy)}{(q^kzx;q)_{\infty}}\\
        &=\frac{(a;q)_{\infty}}{(zx,b;q)_{\infty}}\sum_{k=0}^{\infty}\frac{(b/a;q)_{k}(zx;q)_{k}}{(q;q)_{k}}a^k\RR_{q}(q^kzy).
    \end{align*}
\end{proof}

\section{Mehler's formulas for $\s_{n}(x,y;q)$}

\begin{theorem}
    \begin{align*}
        &\sum_{n=0}^{\infty}\s_{n}(x,y;q)\s_{n}(w,z;q)\frac{t^n}{(q;q)_{n}}\\
        &\hspace{1cm}=\frac{1}{(twx;q)_{\infty}}\sum_{k=0}^{\infty}\frac{q^{2k^2}(twx;q)_{k}}{(q;q)_{k}}(tyz)^k\RR_{q}(tzq^{2k}x)\RR_{q}(tyq^{2k}w)
    \end{align*}
\end{theorem}
\begin{proof}
    \begin{align*}
        &\sum_{n=0}^{\infty}\s_{n}(x,y;q)\s_{n}(w,z;q)\frac{t^n}{(q;q)_{n}}=\rr(yD_{q})\left\{\sum_{n=0}^{\infty}\s_{n}(w,z;q)\frac{(tx)^n}{(q;q)_{n}}\right\}\\
        &=\rr(yD_{q})\left\{\frac{\RR_{q}(tzx)}{(twx;q)_{\infty}}\right\}\\
        &=\sum_{n=0}^{\infty}q^{n^2}\frac{y^n}{(q;q)_{n}}D_{q}^{n}\left\{\frac{\RR_{q}(tzx)}{(twx;q)_{\infty}}\right\}\\
        &=\sum_{n=0}^{\infty}q^{n^2}\frac{y^n}{(q;q)_{n}}\sum_{k=0}^{n}\qbinom{n}{k}_{q}q^{k(k-n)}D_{q}^{k}\{\RR_{q}(tzx)\}D_{q}^{n-k}\left\{\frac{1}{(twq^kx;q)_{\infty}}\right\}\\
        &=\sum_{n=0}^{\infty}q^{n^2}\frac{y^n}{(q;q)_{n}}\sum_{k=0}^{n}\qbinom{n}{k}_{q}q^{k(k-n)}(tz)^kq^{k^2}\RR_{q}(tzq^{2k}x)\frac{(twq^k)^{n-k}}{(twq^kx;q)_{\infty}}\\
        &=\frac{1}{(twx;q)_{\infty}}\sum_{k=0}^{\infty}\frac{q^{2k^2}(twx;q)_{k}}{(q;q)_{k}}(tyz)^k\RR_{q}(tzq^{2k}x)\sum_{n=0}^{\infty}q^{n^2}\frac{(q^{2k}ty)^nw^{n}}{(q;q)_{n}}\\
        &=\frac{1}{(twx;q)_{\infty}}\sum_{k=0}^{\infty}\frac{q^{2k^2}(twx;q)_{k}}{(q;q)_{k}}(tyz)^k\RR_{q}(tzq^{2k}x)\RR_{q}(tyq^{2k}w)
    \end{align*}
\end{proof}

\begin{theorem}
    \begin{align*}
        &\rr(yD_{q})\{(ax,bx;q)_{\infty}\}\\
        &\hspace{1cm}=(ax,bx;q)_{\infty}\sum_{k=0}^{\infty}\frac{q^{3\binom{k}{2}}(qay)^k}{(ax;q)_{k}(bx;q)_{k}(q;q)_{k}}{}_{0}\phi_{2}\left(
    \begin{array}{c}
         -\\
         bq^kx,0
    \end{array}
    ;q,-q^{2k+1}by
    \right).
    \end{align*}
\end{theorem}
\begin{proof}
    \begin{align*}
        &\rr(yD_{q})\{(ax,bx;q)_{\infty}\}=\sum_{n=0}^{\infty}q^{n^2}\frac{y^n}{(q;q)_{n}}D_{q}^{n}\{(ax,bx;q)_{\infty}\}\\
        &=(ax;q)_{\infty}\sum_{n=0}^{\infty}q^{n^2}\frac{y^n}{(q;q)_{n}}q^{\binom{n}{2}}(bq^nx;q)_{\infty}\sum_{k=0}^{n}\qbinom{n}{k}_{q}q^{k(k-n)}\frac{a^kb^{n-k}}{(ax;q)_{k}}\\
        &=(ax,bx;q)_{\infty}\sum_{n=0}^{\infty}q^{n^2}\frac{y^n}{(q;q)_{n}(bx;q)_{n}}q^{\binom{n}{2}}\sum_{k=0}^{n}\qbinom{n}{k}_{q}q^{k(k-n)}\frac{a^kb^{n-k}}{(ax;q)_{k}}\\
        &=(ax,bx;q)_{\infty}\sum_{k=0}^{\infty}\frac{a^k}{(ax;q)_{k}(q;q)_{k}}\sum_{n=k}^{\infty}q^{n^2}\frac{y^nb^{n-k}}{(q;q)_{n-k}(bx;q)_{n}}q^{\binom{n}{2}}q^{k(k-n)}\\
        &=(ax,bx;q)_{\infty}\sum_{k=0}^{\infty}\frac{q^{\binom{k}{2}+k^2}(ay)^k}{(ax;q)_{k}(bx;q)_{k}(q;q)_{k}}\sum_{n=0}^{\infty}q^{\binom{n}{2}+n^2}\frac{(q^{2k}by)^{n}}{(q;q)_{n}(bq^kx;q)_{n}}\\
        &=(ax,bx;q)_{\infty}\sum_{k=0}^{\infty}\frac{q^{3\binom{k}{2}}(qay)^k}{(ax;q)_{k}(bx;q)_{k}(q;q)_{k}}{}_{0}\phi_{2}\left(
    \begin{array}{c}
         -\\
         bq^kx,0
    \end{array}
    ;q,-q^{2k+1}by
    \right).
    \end{align*}
\end{proof}

\begin{theorem}
    \begin{align*}
        &\sum_{n=0}^{\infty}(-1)^nq^{\binom{n}{2}}\s_{n}(x,y;q)\s_{n}(a,bq^{-n};q)\frac{z^n}{(q;q)_{n}}\\
        &=(azx,bzx;q)_{\infty}\sum_{k=0}^{\infty}\frac{q^{3\binom{k}{2}}(qay)^k}{(azx;q)_{k}(bzx;q)_{k}(q;q)_{k}}{}_{0}\phi_{2}\left(
    \begin{array}{c}
         -\\
         bq^kzx,0
    \end{array}
    ;q,-q^{2k+1}by
    \right).
    \end{align*}
\end{theorem}
\begin{proof}
    \begin{align*}
        &\sum_{n=0}^{\infty}(-1)^nq^{\binom{n}{2}}\s_{n}(a,b;q)\s_{n}(x,yq^{-n};q)\frac{z^n}{(q;q)_{n}}\\
        &=\sum_{n=0}^{\infty}(-1)^nq^{\binom{n}{2}}\rr(yD_{q})\{x^n\}\s_{n}(a,bq^{-n};q)\frac{z^n}{(q;q)_{n}}\\
        &=\rr(yD_{q})\left\{\sum_{n=0}^{\infty}(-1)^nq^{\binom{n}{2}}\s_{n}(a,bq^{-n};q)\frac{(xz)^n}{(q;q)_{n}}\right\}\\
        &=\rr(yD_{q})\{(azx,bzx;q)_{\infty}\}\\
        &=(azx,bzx;q)_{\infty}\sum_{k=0}^{\infty}\frac{q^{3\binom{k}{2}}(qay)^k}{(azx;q)_{k}(bzx;q)_{k}(q;q)_{k}}{}_{0}\phi_{2}\left(
    \begin{array}{c}
         -\\
         bq^kzx,0
    \end{array}
    ;q,-q^{2k+1}by
    \right).
    \end{align*}
\end{proof}

\section{Rogers formulas for $\s_{n}(x,y;q)$}

\begin{theorem}
    \begin{align*}
        \sum_{n=0}^{\infty}\sum_{m=0}^{\infty}(-1)^{n}q^{\binom{n}{2}}\s_{n+m}(x,y;q)\frac{t^n}{(q;q)_{n}}\frac{s^m}{(q;q)_{m}}=\frac{(tx;q)_{\infty}}{(sx;q)_{\infty}}{}_{1}\phi_{2}\left(
        \begin{array}{c}
         t/s\\
         tx,0
        \end{array}
        ;q,qy
        \right).
    \end{align*}
\end{theorem}
\begin{proof}
    \begin{align*}
        &\sum_{n=0}^{\infty}\sum_{m=0}^{\infty}(-1)^{n}q^{\binom{n}{2}}\s_{n+m}(x,y;q)\frac{t^n}{(q;q)_{n}}\frac{s^m}{(q;q)_{m}}\\
        &\hspace{1cm}=\sum_{n=0}^{\infty}\sum_{m=0}^{\infty}(-1)^{n}q^{\binom{n}{2}}\rr(yD_{q})\{x^{n+m}\}\frac{t^n}{(q;q)_{n}}\frac{s^m}{(q;q)_{m}}\\
        &\hspace{1cm}=\rr(yD_{q})\left\{\sum_{n=0}^{\infty}\sum_{m=0}^{\infty}(-1)^{n}q^{\binom{n}{2}}\frac{(tx)^n}{(q;q)_{n}}\frac{(sx)^m}{(q;q)_{m}}\right\}\\
        &\hspace{1cm}=\rr(yD_{q})\left\{\sum_{n=0}^{\infty}(-1)^{n}q^{\binom{n}{2}}\frac{(tx)^n}{(q;q)_{n}}\sum_{m=0}^{\infty}\frac{(sx)^m}{(q;q)_{m}}\right\}\\
        &\hspace{1cm}=\rr(yD_{q})\left\{\frac{(tx;q)_{\infty}}{(sx;q)_{\infty}}\right\}\\
        &\hspace{1cm}=\frac{(tx;q)_{\infty}}{(sx;q)_{\infty}}{}_{1}\phi_{2}\left(
        \begin{array}{c}
         t/s\\
         tx,0
        \end{array}
        ;q,qy
        \right).
    \end{align*}
\end{proof}

\begin{theorem}
    \begin{align*}
        \sum_{n=0}^{\infty}\sum_{m=0}^{\infty}\s_{n+m}(x,y;q)\frac{t^n}{(q;q)_{n}}\frac{s^m}{(q;q)_{m}}=\frac{1}{(tx,sx;q)_{\infty}}\sum_{i=0}^{\infty}q^{i^2}\frac{(sx;q)_{i}}{(q;q)_{i}}(ty)^i\RR_{q}(sq^{2i}y).
    \end{align*}
\end{theorem}
\begin{proof}
    \begin{align*}
        \sum_{n=0}^{\infty}\sum_{m=0}^{\infty}\s_{n+m}(x,y;q)\frac{t^n}{(q;q)_{n}}\frac{s^m}{(q;q)_{m}}&=\sum_{n=0}^{\infty}\sum_{m=0}^{\infty}\rr(yD_{q})\{x^{n+m}\}\frac{t^n}{(q;q)_{n}}\frac{s^m}{(q;q)_{m}}\\
        &=\rr(yD_{q})\left\{\sum_{n=0}^{\infty}\sum_{m=0}^{\infty}\frac{(tx)^n}{(q;q)_{n}}\frac{(sx)^m}{(q;q)_{m}}\right\}\\
        &=\rr(yD_{q})\left\{\sum_{n=0}^{\infty}\frac{(tx)^n}{(q;q)_{n}}\sum_{m=0}^{\infty}\frac{(sx)^m}{(q;q)_{m}}\right\}\\
        &=\rr(yD_{q})\left\{\frac{1}{(tx,sx;q)_{\infty}}\right\}\\
        &=\frac{1}{(tx,sx;q)_{\infty}}\sum_{i=0}^{\infty}q^{i^2}\frac{(sx;q)_{i}}{(q;q)_{i}}(ty)^i\RR_{q}(sq^{2i}y).
    \end{align*}
\end{proof}

\Addresses


\begin{thebibliography}{99}

\bibitem{gasper}
G. Gasper, M. Rahman,
Basic Hypergeometric Series, $2^{nd}$ ed., 
Cambridge University Press, Cambridge, MA, 1990.

\bibitem{saad}
H. L. Saad and M.A. Abdul, 
The $q$-Exponential Operator and Generalized Rogers-Szego Polynomials, 
Journal of Advances in Mathematics, 8(1), (2014) 1440–-1455. 
https://doi.org/10.24297/jam.v8i1.6912

\bibitem{Sri}
H.M. Srivastava, A.K. Agarwal,
Generating functions for a class of $q$-polynomials. 
Ann. Mat. Pura Appl. (Ser. 4) \textbf{154} (1989) 99--109. https://doi.org/10.1007/BF01790345

\bibitem{garrett}
K. Garrett, M.E.H. Ismail, and D Stanton, 
Variants of the Rogers–Ramanujan Identities, 
Advances in Applied Mathematics, Volume 23, Issue 3, 1999, 274--299,
https://doi.org/10.1006/aama.1999.0658.

\bibitem{orozco}
R. Orozco, 
Deformed Rogers-Szeg\"{o} Polynomials and the Generalized $q$-Exponential Operator, 
arXiv:2306.07431v4, 
(2024).

\end{thebibliography}
\end{document}